\documentclass[reqno]{amsart}
\usepackage[all]{xy}
\usepackage{amssymb}
\usepackage{enumitem}
\usepackage{hyperref}
\usepackage{color}

\input xy
\xyoption{all}
\CompileMatrices

\definecolor{myurlcolor}{rgb}{0,0,0.5}
\hypersetup{colorlinks,
linkcolor=myurlcolor,
citecolor=myurlcolor,
urlcolor=myurlcolor}

\newcommand{\beq}{\begin{displaymath}}
\newcommand{\eeq}{\end{displaymath}}
\newcommand{\beqn}{\begin{equation}}
\newcommand{\eeqn}{\end{equation}}

\newcommand{\ra}{\rightarrow}

\newcommand{\ind}[1]{\textbf{#1\index{#1}}}

\swapnumbers
\newtheorem{prop}{Proposition}[section]
\newtheorem{thm}[prop]{Theorem}
\newtheorem*{thm*}{Theorem}
\newtheorem{defn}[prop]{Definition}
\newtheorem{cor}[prop]{Corollary}
\newtheorem{lem}[prop]{Lemma}
\theoremstyle{definition} 

\newtheorem{rem}[prop]{Remark}

\begin{document}

\title{Categories of Fractions Revisited}

\author{Tobias Fritz}
\email{tobias.fritz@icfo.es}
\address{ICFO -- Institut de Ci\`{e}ncies Fot\`{o}niques\\ 
Mediterranean Technology Park\\ 
08860 Castelldefels (Barcelona)\\ 
Spain}

\begin{abstract}
The theory of categories of fractions as originally developed by Gabriel and Zisman~\cite{GabZis} is reviewed in a pedagogical manner giving detailed proofs of all statements. A weakening of the category of fractions axioms used by Higson~\cite{Hig} is discussed and shown to be equivalent to the original axioms.
\end{abstract}

\maketitle
\tableofcontents

\section{Introduction}

In category theory, the concept of localization is a tool for constructing a new category from a given one. The idea is as follows: a category may have a certain class of morphisms which are not all invertible, although morally they ``should'' be invertible. As an example, one may consider weak homotopy equivalences in the homotopy category of topological spaces: some weak homotopy equivalences are homotopy equivalences, and hence isomorphisms, but not all of them are~\cite{Hatcher}; on the other hand, two weakly homotopy equivalent spaces behave in absolutely the same way concerning the properties probed by maps from or to suitably nice spaces, and hence should morally be isomorphic.

Given such a class of morphisms in a category, one can form a \emph{localization} of the original category, which is a new category which guarantees all ``morally invertible'' morphisms to be invertible, while approximating the original category as closely as possible. This idea can be made precise in terms of a universal property; see section~\ref{sectionlocalization}.

Localizations exist not only for categories, but also for other kinds of algebraic structures. For example for rings: adjoining formal inverses for a certain class of ring elements yields a new ring from a given one. Under certain conditions on the class $\mathcal{W}$ of elements to be inverted---the so-called Ore conditions---there is a particularly nice way to describe the elements of the localized ring in terms of an equivalence class of formal fractions, where a formal fraction is defined to have an element of the original ring in the numerator and an element of $\mathcal{W}$ in the denominator.

It turns out that pretty much the same technique that works for rings can also applied to categories. Under certain conditions, the localization of a category with respect to a class of morphisms can be described in terms of ``formal fractions''. If this construction is possible, the resulting localization is a \emph{category of fractions}. In some cases, such an abstract construction can be more useful than a concrete (in the category-theoretical sense!) description of the localization. Furthermore, categories of fractions can be relevant for other general categorical constructions; the theory of Verdier localization in the context of triangulated categories is an example.

Due to the metamathematical nature of category theory, the objectives in category theory are quite different from those in ring theory: thinking of a category as representing the collection of models of a mathematical theory, taking a category of fractions is a tool to construct a new \emph{mathematical theory} from a given one. 

\subsection{Summary.} In section~\ref{sectionlocalization}, the concept of localization of a category is introduced and compared to taking a quotient category. Section~\ref{catsoffractions} then gives a detailed account of the category of fraction axioms and their consequences; in particular, all proofs are presented in complete detail. Section~\ref{sectionweakening} goes on to study a weakening of the category of fraction axioms which was originally introduced by Higson~\cite{Hig} in the context of bivariant $K$-theory of $C^*$-algebras. It is shown that this weakening is equivalent to the usual set of axioms. This is the only new result of the present work. Finally, section~\ref{sectionadditive} shows that a category of fractions is additive in case the original category is additive.

\subsection{Notation and terminology.} In all commutative diagrams, the objects are simply denoted by fat dots ``$\bullet$''. Unless noted otherwise, all diagrams commute. Idenitity morphisms are pictured as double lines ``$\xymatrix{\ar@{=}[r]&}$''. The words ``isomorphism'' and ``monomorphism'' are abbreviated respectively as ``iso'' and ``mono''. A split mono is a morphism which has a left inverse; it automatically is a mono. Domain and codomain of a morphism $f$ are written as $\mathrm{dom}(f)$ and $\mathrm{cod}(f)$, respectively.

This article is a revised version of part of the author's Master's thesis written at the University of M\"unster in 2007.

\section{Localization of categories}
\label{sectionlocalization}

In some contexts it may happen that we have a category $\mathcal{C}$ which is -- in a sense depending on the situation -- not well-behaved. For example, it might be that it is too hard to do concrete calculations, or it might be that $\mathcal{C}$ does not have some desired formal property. Then one can try to find a second category $\widehat{\mathcal{C}}$ which has the same objects as $\mathcal{C}$ together with a functor $\mathcal{C}\ra\widehat{\mathcal{C}}$ which is the identity on objects, such that $\widehat{\mathcal{C}}$ is better-behaved and approximates $\mathcal{C}$ in some appropriate sense also depending on the situation. Then instead of working in $\mathcal{C}$ directly, one can transport the morphisms from $\mathcal{C}$ to $\widehat{\mathcal{C}}$ via the functor $\mathcal{C}\ra\widehat{\mathcal{C}}$ and prove theorems about the morphisms in the well-behaved category $\widehat{\mathcal{C}}$. The price one has to pay is that in general some information about the structure of $\mathcal{C}$ is lost on the way.

Now there are at least two concrete ways to make this precise. The first one is the notion of a \ind{quotient category}. Suppose we are given an equivalence relation $\sim$ on every morphism set $\mathcal{C}(A,B)$ which is preserved under composition, meaning that
\beqn
(f_1\sim f_2)\Longrightarrow (f_1g\sim f_2g) \wedge (hf_1\sim hf_2)\qquad\forall f_1,f_2,g,h\in\mathcal{C}
\eeqn
whenever these compositions are defined. Then composition of equivalence classes is well-defined and defines the quotient category $\mathcal{C}/\!\!\sim$ together with the canonical projection functor $\mathcal{C}\ra\mathcal{C}/\!\!\sim$. Any kind of homotopy theory serves as a good example.

The second way is a concept called \ind{localization}. It may be familiar from ring theory. Suppose we are given a category $\mathcal{C}$ and a subclass of morphisms called $\mathcal{W}$, which ``morally'' ought to be isos, but in $\mathcal{C}$ not necessarily all of them are; using the letter $\mathcal{W}$ is supposed to suggest a reading like ``weak equivalence''~\cite{Hov}. We try to turn all the morphisms in $\mathcal{W}$ into isos by adjoining formal inverses for them. More precisely, we are looking for a category $\widehat{\mathcal{C}}=\mathcal{C}[\mathcal{W}^{-1}]$ equipped with a localization functor $Loc:\mathcal{C}\ra\mathcal{C}[\mathcal{W}^{-1}]$ which has the following universal property:
\begin{enumerate}[label=(\alph{enumi})]
\item $Loc(w)$ is an iso for all $w\in\mathcal{W}$,
\item\label{b} If $F:\mathcal{C}\ra\mathcal{D}$ is any functor which maps $\mathcal{W}$ to isos, then $F$ factors uniquely over $Loc$ as in the diagram
\beqn
\begin{split}
\label{univ}
\xymatrix{{}\mathcal{C}\ar[rr]^{Loc}\ar[rd]_F & & {}\mathcal{C}[\mathcal{W}^{-1}]\ar@{-->}[ld]^{\exists!} \\
& {}\mathcal{D} & }
\end{split}
\eeqn
\end{enumerate}

In case such a functor exists, the category $\mathcal{C}[\mathcal{W}^{-1}]$ is called the ``localization of $\mathcal{C}$ with respect to $\mathcal{W}$''. It serves as the desired approximation $\widehat{\mathcal{C}}$ to $\mathcal{C}$.

Since $\mathcal{C}[\mathcal{W}^{-1}]$ is defined via a universal property, it is certainly unique (up to a unique iso). Proving existence is the nontrivial part.

\begin{thm}
$\mathcal{C}[\mathcal{W}^{-1}]$ and $Loc$ always exist.
\end{thm}

\begin{proof}(from~\cite[III.2.2]{GM} and~\cite[1.1]{GabZis}).
The category $\mathcal{C}[\mathcal{W}^{-1}]$ can be constructed in two steps: start with the category of paths---call it $\mathcal{P}(\mathcal{C},\mathcal{W}^{-1})$---which has as objects the objects of $\mathcal{C}$, and as morphisms finite strings $\langle l_1,\ldots, l_n\rangle$ of composable literals, where a literal $l_k$ is either a morphism of $\mathcal{C}$ (including $\mathcal{W}$) or a formal inverse of a morphism in $\mathcal{W}$. Composition of these morphisms is defined as concatenation of strings. For every object $A\in\mathcal{C}$, we also have the empty string $\langle\rangle_A$ which starts and ends at $A$ and is the identity morphism of $A$ in $\mathcal{P}(\mathcal{C},\mathcal{W}^{-1})$. This whole definition can be summarized by saying that $\mathcal{P}(\mathcal{C},\mathcal{W}^{-1})$ is the free category generated by the graph $\mathcal{C}\cup\mathcal{W}^{-1}$.

There is a canonical map $\mathcal{C}\ra\mathcal{P}(\mathcal{C},\mathcal{W}^{-1})$ which is the identity on objects and maps every morphism $f\in\mathcal{C}$ to the corresponding single-literal string $\langle f\rangle$. This map already has the desired universal property~\ref{b}. However, neither is this map a functor nor does it map $\mathcal{W}$ to isos. We can easily fix both of these issues by taking a quotient category of $\mathcal{P}(\mathcal{C},\mathcal{W}^{-1})$ in which these properties are enforced. To this end, we introduce the equivalence relation $\sim$ on strings generated by closure under composition together with the elementary equivalences
\begin{enumerate}[label=(\alph{enumi})]
\item $\langle\,\rangle_A\sim\langle\,\mathrm{id}_A\rangle\qquad \forall A\in\mathrm{Obj}(\mathcal{C})$,
\item $\langle g,f\rangle\sim \langle gf\rangle\qquad   \forall f,g\in\mathcal{C}$ for which the composition $gf$ exists,
\item $\langle w,w^{-1}\rangle\sim\langle\,\rangle_{\mathrm{cod}(w)}\,,\quad \langle w^{-1},w\rangle\sim\langle\,\rangle_{\mathrm{dom}(w)}\qquad\forall w\in\mathcal{W}$.
\end{enumerate}
Then it is clear that the induced map $Loc:\mathcal{C}\ra\mathcal{P}(\mathcal{C},\mathcal{W}^{-1})/\!\!\sim$ is a functor and maps $\mathcal{W}$ to isos. 

As for universality, suppose we are given some functor $F:\mathcal{C}\ra\mathcal{D}$ mapping $\mathcal{W}$ to isos. It induces a unique functor $\mathcal{P}(\mathcal{C},\mathcal{W}^{-1})\ra\mathcal{D}$. This functor maps the above elementary equivalences to equalities, thus uniquely factors over the quotient category $\mathcal{P}(\mathcal{C},\mathcal{W}^{-1})/\!\!\sim$.
\end{proof}

\begin{rem}
\begin{enumerate}[label=(\alph{enumi})]
\item For locally small $\mathcal{C}$, the localization $\mathcal{C}[\mathcal{W}^{-1}]$ need not be locally small. Even under the conditions to be discussed in the next section, it may well happen that the localization has proper classes as the collections of morphisms between some pairs of objects. Showing that this does not happen in a concrete case seems to be a hard problem; one case where local smallness is known is for model categories and localizing with respect to the class of weak equivalences (see~\cite[p.7 and 1.2.10]{Hov}).
\item The canonical functor to a quotient category $\mathcal{C}\ra\mathcal{C}/\!\!\sim$ is full by definition of $\mathcal{C}/\!\!\sim$. However, this is usually not true for a localization functor $Loc:\mathcal{C}\ra\mathcal{C}[\mathcal{W}^{-1}]$.
\end{enumerate}
\end{rem}

\section{Categories of fractions}
\label{catsoffractions}

In all diagrams dealing with categories of fractions, a wiggly arrow $\xymatrix{\ar@{~>}[r]&}$ denotes a morphism in $\mathcal{W}$, while a straight arrow $\xymatrix{\ar[r]&}$ is any morphism of $\mathcal{C}$.

In certain situations, the localization $\mathcal{C}[\mathcal{W}^{-1}]$ can be described much more explicitly, which implies a large gain of control over the structure of this category. We say that the pair $\mathcal{W}\subseteq\mathcal{C}$ allows a \ind{calculus of left fractions}, if the following conditions are satisfied:
\\
\begin{enumerate}[label=(L\theenumi)]
\setcounter{enumi}{-1}
\item\label{L0} $\mathcal{W}$ contains all identity morphisms and is closed under composition. In other words, $\mathcal{W}\subseteq\mathcal{C}$ is a subcategory containing all objects.
\item\label{L1}
Given any $w\in\mathcal{W}$ and an arbitrary morphism $f$ with $\mathrm{dom}(f)=\mathrm{dom}(w)$, we can find $w'\in\mathcal{W}$ with $\mathrm{dom}(w')=\mathrm{cod}(f)$ and some morphism $f'$ with $\mathrm{cod}(f')=\mathrm{cod}(w')$, such that the diagram
\beq
\xymatrix{{}{\bullet}\ar@{~>}[r]^w\ar[d]_f & {}{\bullet}\ar[d]^{f'} \\
{\bullet}\ar@{~>}[r]_{w'} & {\bullet}}
\eeq
commutes.
\item\label{L2}
Given $w\in\mathcal{W}$ and parallel morphisms $f_1$, $f_2$ such that $f_1w=f_2w$, there exists $w'\in\mathcal{W}$ such that $w'f_1=w'f_2$.
\beq
\xymatrix{{}{\bullet}\ar@{~>}[r]^w & {}{\bullet}\ar@/^/[r]^{f_1}\ar@/_/[r]_{f_2} & {}{\bullet}\ar@{~>}[r]^{w'} & {}{\bullet}}
\eeq
\end{enumerate}

These conditions are exact analogues of the Ore conditions in the theory of (not necessarily commutative) rings~\cite[p. 3]{Jate}.

\begin{rem}
\label{generated}
Condition~\ref{L0} is not an essential restriction: if~\ref{L1} and~\ref{L2} hold for some class of morphisms $\mathcal{W}$, then both also hold for the $\mathcal{C}$-subcategory generated by $\mathcal{W}\cup\{\mathrm{id}_A,\,A\in\mathrm{Obj}(\mathcal{C})\}$. Hence $\mathcal{W}$ can be replaced by this subcategory.
\end{rem}

\begin{proof}
We assume that $\mathcal{W}$ satisfies~\ref{L1} and~\ref{L2}, but not necessarily~\ref{L0}. Then $\mathcal{W}\cup\{\mathrm{id}_A,\,A\in\mathrm{Obj}(\mathcal{C})\}$ certainly also satisfies~\ref{L1} and~\ref{L2}, so it is enough to show that closing $\mathcal{W}$ under composition gives a morphism class $\widehat{\mathcal{W}}$ which also satisfies~\ref{L1} and~\ref{L2}.

Let $w_1, w_2\in\mathcal{W}$ be composable to $w=w_1w_2$. Given any $f$ with $\mathrm{dom}(f)=\mathrm{dom}(w_1)$, applying~\ref{L1} twice shows that we can find $w'_1,w'_2\in\mathcal{W}$ and $f',f''\in\mathcal{C}$ such that the diagram
$$
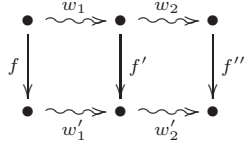
\xymatrix{ {\bullet}\ar@{~>}[r]^{w_1} \ar[d]_f & \bullet \ar@{~>}[r]^{w_2} \ar[d]^{f'} & \bullet\ar[d]^{f''} \\
\bullet \ar@{~>}[r]_{w'_1} & \bullet \ar@{~>}[r]_{w'_2} & \bullet }
$$
commutes. Now $w'=w'_1w'_2\in\widehat{\mathcal{W}}$ and $f''$ have the required properties with respect to $w=w_1w_2$ and $f$. Applying this argument inductively proves the claim about~\ref{L1}.

Concerning~\ref{L2}, we similarly consider the situation $f_1w_2w_1=f_2w_2w_1$, and obtain
$$
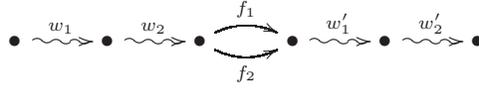
\xymatrix{\bullet \ar@{~>}[r]^{w_1} & \bullet \ar@{~>}[r]^{w_2} & \bullet\ar@/^/[r]^{f_1}\ar@/_/[r]_{f_2} & \bullet\ar@{~>}[r]^{w'_1} & \bullet\ar@{~>}[r]^{w'_2} & \bullet }
$$
where, thanks to~\ref{L2}, we could choose $w'_1$ such that $w'_1f_1w_2=w'_1f_2w_2$, and then $w'_2$ such that $w'_2w'_1f_1=w'_2w'_1f_2$, as desired.
\end{proof}

\begin{defn}
\label{roof}
A \ind{roof} $(f,w)$ between two objects $\mathrm{dom}(f)$ and $\mathrm{dom}(w)$ is a diagram of the form
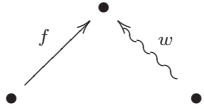
\beq
\xymatrix{ & {}{\bullet} & \\{}
{\bullet}\ar[ur]^f & & {}{\bullet}\ar@{~>}[ul]_w}
\eeq
\end{defn}

From now on, we assume that $\mathcal{W}\subseteq\mathcal{C}$ satisfies~\ref{L0},~\ref{L1} and~\ref{L2}, and derive some consequences from this assumption.

The way to think of a roof $(f,w)$ is as being a formal ``left fraction'' $w^{-1}f$, defining a formal morphism from the lower left object to the lower right object. Then~\ref{L1} intuitively states that it is possible to turn any formal ``right fraction'' $fw^{-1}$ into a left fraction $w'^{-1}f'$, since $w'f=f'w$ together with invertibility of $w$ and $w'$ implies $fw^{-1}=w'^{-1}f'$.

\begin{defn}
\label{roofeq}
Two roofs $(f_1,w_1)$ and $(f_2,w_2)$ are \ind{equivalent} if there are morphisms $g$ and $h$ forming a third roof $(gf_1,gw_1)=(hf_2,hw_2)$ as in the diagram
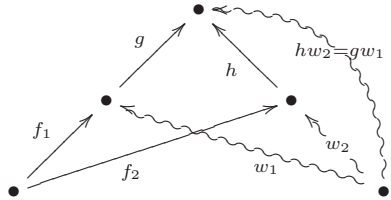
\beq
\xymatrix{ & & {}{\bullet} & & \\
& {}{\bullet}\ar[ur]^g & & {}{\bullet}\ar[ul]^h & \\{}
{\bullet}\ar[ur]^{f_1}\ar[urrr]_(.4){f_2} & & & & {}{\bullet} \ar@{~>}[ulll]^(.4){w_1} \ar@{~>}[ul]|{w_2} \ar@{~>}@/_2.2pc/[uull]|{hw_2=gw_1}}
\eeq
\end{defn}

Note that it is not required that $g$ or $h$ be an element of $\mathcal{W}$, only the \emph{composition} $gw_1=hw_2$ has to be in $\mathcal{W}$. The equality of $(gf_1,gw_1)=(hf_2,hw_2)$ is expressed by commutativity of the two squares in the diagram.

The goal of this section is to establish that equivalence classes of roofs form a category under the appropriate composition operation, and that this category is the localization $\mathcal{C}[\mathcal{W}^{-1}]$. This will be done in a sequence of small steps. Intuitively, the first step is to show that the roof $(f',w')$ one obtains from using~\ref{L1} to turn a formal right fraction $fw^{-1}$ into a formal left fraction $w'^{-1}f'$ is unique up to equivalence. This will let us define composition of equivalence classes of roofs later on.

\begin{lem}
\label{l1unique}
Any two ways to choose $f'$ and $w'$ in~\ref{L1} define equivalent roofs.
\end{lem}

\begin{proof}
Imagine two possible choices $(f'_1,w'_1)$ and $(f'_2,w'_2)$ as in the partially commutative diagram
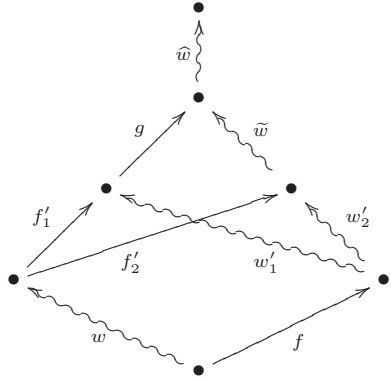
\beq
\xymatrix{&&{\bullet} &&\\
&&{\bullet}\ar@{~>}[u]^{\widehat{w}}\\
&{\bullet}\ar[ru]^g && {\bullet}\ar@{~>}[lu]_{\widetilde{w}} &\\
{\bullet}\ar[ru]^{f'_1}\ar[rrru]_(.4){f'_2}&&&&{\bullet}\ar@{~>}[lllu]^(.4){w'_1}\ar@{~>}[lu]_{w'_2}\\
&&{\bullet}\ar@{~>}[llu]^w\ar[rru]_f &&}
\eeq
By~\ref{L1}, $g$ and $\widetilde{w}$ were chosen such that $gw'_1=\widetilde{w}w'_2$. This is not yet an equivalence of roofs, since, in general, $gf'_1\neq\widetilde{w}f'_2$. However, we do know that $gf'_1w=\widetilde{w}f'_2w$, so by~\ref{L2} we can choose $\widehat{w}$ such that $\widehat{w}gf'_1=\widehat{w}\widetilde{w}f'_2$. This makes $(f'_1,w'_1)$ and $(f'_2,w'_2)$ equivalent via $\widehat{w}g$ and $\widehat{w}\widetilde{w}$.
\end{proof}

\begin{lem}
\label{trans}
The equivalence of roofs from definition~\ref{roofeq} is an equivalence relation.
\end{lem}

\begin{proof}
Reflexivity and symmetry are obvious. For transitivity, suppose we are given an equivalence between $(f_1,w_1)$ and $(f_2,w_2)$, and one between $(f_2,w_2)$ and $(f_3,w_3)$, as in the partially commutative diagram
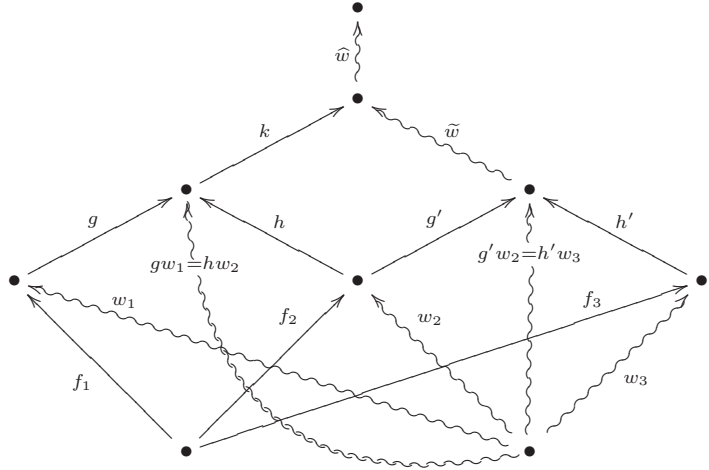
\beq
\xymatrix{&&&&{}{\bullet}&&&&\\
&&&&{}{\bullet}\ar@{~>}[u]^{\widehat{w}}&&&&\\
&&{}{\bullet}\ar[rru]^k&&&&{}{\bullet}\ar@{~>}[llu]_{\widetilde{w}}&&\\
{\bullet}\ar[rru]^g&&&&{\bullet}\ar[llu]_h\ar[rru]^{g'}&&&&{\bullet}\ar[llu]_{h'}\\\\
&&{\bullet}\ar[lluu]^{f_1}\ar[rruu]^(.7){f_2}\ar[rrrrrruu]^(.8){f_3}&&&&{\bullet}\ar@{~>}[lllllluu]_(.8){w_1}\ar@{~>}[lluu]_(.7){w_2}\ar@{~>}[rruu]_{w_3}\ar@{~>}@/^4pc/[lllluuu]|(.88){gw_1=hw_2}\ar@{~>}[uuu]|(.76){g'w_2=h'w_3}&&}
\eeq
Here, the equivalence between $(f_1,w_1)$ and $(f_2,w_2)$ is assumed to be implemented by $g$ and $h$, while the one between $(f_2,w_2)$ and $(f_3,w_3)$ is implemented by $g'$ and $h'$. The commutativity conditions for the two equivalences are
\beqn
\label{comrel}
gf_1 = hf_2,\quad gw_1 = hw_2;\qquad\quad	 g'f_2 = h'f_3,\quad g'w_2=h'w_3
\eeqn
In the upper part of the diagram, $k$ and $\widetilde{w}$ were obtained by applying~\ref{L1} to the two wiggly arrows $gw_1=hw_2$ and $g'w_2=h'w_3$. The corresponding commutativity assertion of~\ref{L1} then is $khw_2=\widetilde{w}g'w_2$. By virtue of~\ref{L2}, we can then find the drawn $\widehat{w}$ such that $\widehat{w}kh=\widehat{w}\widetilde{w}g'$. Together with the relations~(\ref{comrel}), this means that the compositions $\widehat{w}kg$ and $\widehat{w}\widetilde{w}h'$ of the morphisms which go up along the sides implement an equivalence between $(f_1,w_1)$ and $(f_3,w_3)$.
\end{proof}

Under a closer look, this argument is actually a special case of the argument used to prove lemma~\ref{l1unique}. In fact, we could also have applied lemma~\ref{l1unique} directly to the two roofs $(h,hw_2)$ and $(g',g'w_2)$, since both are~\ref{L1}-complements of the formal right fraction $\mathrm{id}_{\mathrm{dom}(w_2)}w_2^{-1}$.

\begin{rem}
\label{roofs2cat}
One can also take a $2$-categorical point of view which gives some more intuitive insight on the notion of equivalence of roofs. We get something resembling a 2-category as follows: on the objects of $\mathcal{C}$ we define a $1$-morphism to be a roof in $\mathcal{C}$ with respect to $\mathcal{W}$. For a roof $(f,w)$, we define $\mathrm{dom}((f,w))=\mathrm{dom}(f)$ and $\mathrm{cod}((f,w))=\mathrm{dom}(w)$. A $2$-morphism from a roof $(f_1,w_1)$ to a parallel roof $(f_2,w_2)$ is then defined to be a commutative diagram
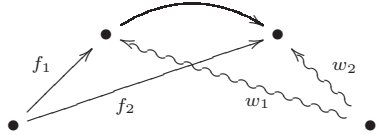
\beq
\xymatrix{&{\bullet}\ar@/^1pc/[rr] && {\bullet}&\\
{\bullet}\ar[ru]^{f_1}\ar[rrru]_(.4){f_2} &&&&{\bullet}\ar@{~>}[lllu]^(.4){w_1}\ar@{~>}[lu]_{w_2}}
\eeq
The existence of such a $2$-morphism makes $(f_1,w_1)$ and $(f_2,w_2)$ equivalent; we call such a $2$-morphism an \ind{elementary equivalence}. A $2$-morphism from $(f_1,w_1)$ to $(f_2,w_2)$ can be composed with a $2$-morphism from $(f_2,w_2)$ to $(f_3,w_3)$. This resembles the vertical composition in a 2-category. Now the observation is that two roofs are equivalent if and only if they can be connected by a finite path of $2$-morphisms, where each $2$-morphism is either traversed from its domain to its codomain or in the reverse direction. To see this, note that the third roof $(gf_1,hw_2)$ in the diagram of definition~\ref{roofeq} is connected to each of the other two roofs by a 2-morphism. The other implication direction follows from the transitivity statement of lemma~\ref{trans} and the fact that two parallel roofs connected by a single $2$-morphism are equivalent. Hence lemma~\ref{trans} can also be reinterpreted as a connectivity statement about the category of parallel roofs between some pair of objects.

In what follows, we will define a (weakly associative) composition of $1$-morphisms. A horizontal composition of $2$-morphisms does not seem to exist in general, although it seems related to the upcoming proof that the composition of $1$-morphisms is well-defined up to equivalence.

We end this remark by pointing out again that this 2-categorical picture is a non-rigorous intuition.
\end{rem}

Lemma~\ref{l1unique} also allows the definition of composition for equivalence classes of roofs:

\begin{defn}
\label{compo}
Given two roofs $(f_1,w_1)$ and $(f_2,w_2)$ which are composable in the sense that $\mathrm{dom}(w_1)=\mathrm{dom}(f_2)$, we define their composition as 
\beq
(f_2,w_2)\circ(f_1,w_1)\equiv (\widetilde{f}f_1,\widetilde{w}w_2)
\eeq
where $\widetilde{f}$ and $\widetilde{w}$ in
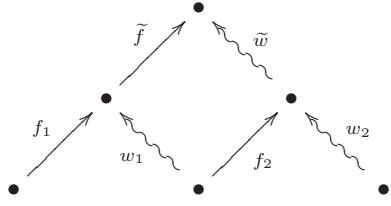
\beq
\xymatrix{&&{\bullet}&&\\
& {\bullet}\ar[ur]^{\widetilde{f}} && {\bullet}\ar@{~>}[lu]_{\widetilde{w}}&\\
{\bullet}\ar[ru]^{f_1} && {\bullet}\ar@{~>}[lu]^{w_1}\ar[ru]_{f_2} && {\bullet}\ar@{~>}[lu]_{w_2}}
\eeq
were obtained by means of~\ref{L1}.
\end{defn}

Thanks to lemma~\ref{l1unique}, the equivalence class of $(\widetilde{f},\widetilde{w})$ is unique. Therefore, so is the equivalence class of $(\widetilde{f}f_1,\widetilde{w}w_2)$.

\begin{lem}
This composition does not depend on the equivalence class of either of the two roofs.
\end{lem}

\begin{proof}
For both pairs of roofs, it is sufficient to consider the case that they are connected by an elementary equivalence as described in remark~\ref{roofs2cat}. Thus suppose we are given the lower half of the diagram
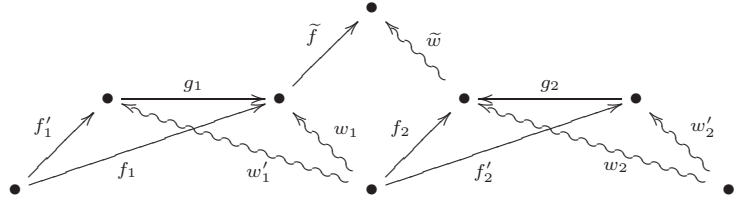
\beq
\xymatrix{&&&&{\bullet}&&&&\\
&{\bullet}\ar[rr]^{g_1} && {\bullet}\ar[ur]^{\widetilde{f}} && {\bullet}\ar@{~>}[lu]_{\widetilde{w}} && {\bullet}\ar[ll]_{g_2} &\\
{\bullet}\ar[ru]^{f'_1}\ar[rrru]_(.4){f_1} &&&& {\bullet}\ar@{~>}[lu]_{w_1}\ar@{~>}[lllu]^(.4){w'_1}\ar[ru]^{f_2}\ar[rrru]_(.4){f'_2} &&&& {\bullet}\ar@{~>}[lu]_{w'_2}\ar@{~>}[lllu]^(.4){w_2}}
\eeq
which represents two pairs of elementarily equivalent roofs. After possible renamings $(f_1,w_1)\leftrightarrow (f'_1,w'_1)$ and $(f_2,w_2)\leftrightarrow (f'_2,w'_2)$, we can assume that $g_1$ goes from $\mathrm{cod}(f'_1)$ to $\mathrm{cod}(f_1)$, while $g_2$ similarly points from $\mathrm{cod}(f'_2)$ to $\mathrm{cod}(f_2)$.

Applying~\ref{L1} to the pair $w_1,f_2$ yields $\widetilde{f}$ and $\widetilde{w}$. Then $(\widetilde{f}f_1,\widetilde{w}w_2)$ is a possible roof representing the composition $(f_2,w_2)\circ(f_1,w_1)$. Similarly, $(\widetilde{f}g_1f'_1,\widetilde{w}g_2w'_2)$ is a possible roof representing the composition $(f'_2,w'_2)\circ(f'_1,w'_1)$. By commutativity, these roofs coincide, so in particular they are equivalent.
\end{proof}

\begin{thm}
\label{loc}
If $\mathcal{W}\subseteq\mathcal{C}$ admits a calculus of left fractions, then the category $\mathcal{C}[\mathcal{W}^{-1}]$ can be described as the category with the same objects as $\mathcal{C}$, morphisms equivalence classes of roofs, and composition as defined above. The localization functor $Loc:\mathcal{C}\ra\mathcal{C}[\mathcal{W}^{-1}]$ is given by $f\mapsto (f,\mathrm{id})$.
\end{thm}

\begin{proof}
Associativity of composition follows from the symbolic diagram
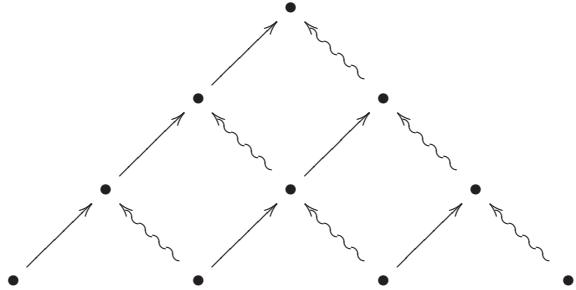
\beq
\xymatrix{&&& {\bullet} &&&\\
&&{\bullet}\ar[ru] && {\bullet}\ar@{~>}[lu] &&\\
&{\bullet}\ar[ur] && {\bullet}\ar@{~>}[lu]\ar[ru] && {\bullet}\ar@{~>}[lu]&\\
{\bullet}\ar[ru] && {\bullet}\ar@{~>}[lu]\ar[ru] && {\bullet}\ar@{~>}[lu]\ar[ru] && {\bullet}\ar@{~>}[lu]}
\eeq
where the three lower roofs are those to be composed; the rest of the diagram is obtained by three applications of~\ref{L1}.
Then the large roof from the left to the right formed by composing the morphisms along the sides is a representative for the composition of the three lower roofs in both possible ways of bracketing. This shows associativity. Furthermore, the equivalence classes of the roofs $(\mathrm{id},\mathrm{id})$ obviously function as identity morphisms. Therefore, taking equivalence classes of roofs as morphisms on $\mathrm{Obj}(\mathcal{C})$ gives a well-defined category $\mathcal{C}[\mathcal{W}^{-1}]$.

Concerning functoriality, $Loc$ preserves identities by definition, and preserves composition by the diagram
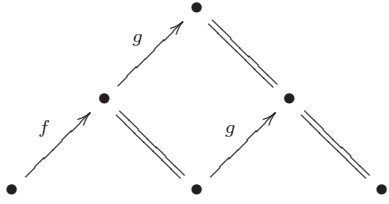
\beq
\xymatrix{&&{\bullet}&&\\
&{\bullet}\ar[ur]^g && {\bullet}\ar@{=}[lu]\\
{\bullet}\ar[ur]^f && {\bullet}\ar@{=}[lu]\ar[ru]^g && {\bullet}\ar@{=}[lu]}
\eeq
which says that the roof $(gf,\mathrm{id})$ is a representative for the equivalence class of $(g,\mathrm{id})\circ(f,\mathrm{id})$.

Under $Loc$, the image of some $w\in\mathcal{W}$ is $(w,\mathrm{id})$, and this image has as its inverse element the class of $(\mathrm{id},w)$ since $(w,w)$ is a representative of both $(\mathrm{id},w)\circ(w,\mathrm{id})$ and $(w,\mathrm{id})\circ(\mathrm{id},w)$, and there is an obvious equivalence $(w,w)\sim (\mathrm{id},\mathrm{id})$. In particular, $Loc$ maps $\mathcal{W}$ to isos.

It remains to check universality. Suppose we have some functor $F:\mathcal{C}\ra\mathcal{D}$ which maps $\mathcal{W}$ to isos. First we need to show that $F$ uniquely extends to roofs. By the desired commutativity of~(\ref{univ}), any such extension has to map the roof $(f,\mathrm{id})$ to $F(f)$. Similarly, since the class $[(\mathrm{id},w)]$ is the inverse of the class $[(w,\mathrm{id})]$, any such extension maps $(\mathrm{id},w)$ to $F(w)^{-1}$. But now $(f,w)$ is a representative of the composition $(\mathrm{id},w)\circ(f,\mathrm{id})$, so we need $(f,w)\mapsto F(w)^{-1}F(f)$.

We still have to check that this assignment is well-defined on equivalence classes and that it is functorial. Consider an elementary equivalence of roofs as in remark~\ref{roofs2cat},
\beq
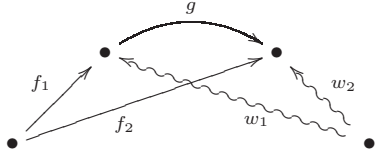
\xymatrix{&{\bullet}\ar@/^1pc/[rr]^g && {\bullet}&\\
{\bullet}\ar[ru]^{f_1}\ar[rrru]_(.4){f_2} &&&&{\bullet}\ar@{~>}[lllu]^(.4){w_1}\ar@{~>}[lu]_{w_2}}
\eeq
Then in $\mathcal{D}$ we have $F(w_2)=F(g)F(w_1)$, so $F(w_1)^{-1}=F(w_2)^{-1}F(g)$. Then the calculation
\beqn
F(w_1)^{-1}F(f_1)=F(w_2)^{-1}F(g)F(f_1)=F(w_2)^{-1}F(f_2)
\eeqn
shows that the equivalent roofs get mapped to identical morphisms in $\mathcal{D}$. 

Functoriality follows by very similar reasoning. Given a pair of composable roofs together with their composition as in definition~\ref{compo}, it holds that 
\beqn
F(\widetilde{f})F(w_1)=F(\widetilde{w})F(f_2)
\eeqn
so that we get
\beqn
\label{commu}
F(\widetilde{w})^{-1}F(\widetilde{f})=F(f_2)F(w_1)^{-1}
\eeqn
Applying first the functor and composing the roofs afterwards yields 
\beqn
\label{comproofs}
\left(F(w_2)^{-1}F(f_2)\right)\circ\left(F(w_1)^{-1}F(f_1)\right)
\eeqn
while for the other direction we end up with $F(\widetilde{w}w_2)^{-1}F(\widetilde{f}f_1)$, which coincides with~(\ref{comproofs}) by~(\ref{commu}) and functoriality of $F$.
\end{proof}

If $\mathcal{W}\subseteq\mathcal{C}$ satisfies~\ref{L0} (which is self-dual) and additionally the conditions (R1) and (R2), which are defined to be the category-theoretic duals of~\ref{L1} and~\ref{L2}, then we say that $\mathcal{W}$ allows a \ind{calculus of right fractions}. In this case, the dual theorem holds: $\mathcal{C}[\mathcal{W}^{-1}]$ can be described in terms of equivalence classes of roofs $(w,f)$ which now represent right fractions $fw^{-1}$. If all five of the (L$\ast$) and (R$\ast$) conditions hold, we say that $\mathcal{W}$ admits a \textbf{calculus of left and right fractions}.

\section{Weakening the requirements}
\label{sectionweakening}

In~\cite{Hig}, a notion of category of fractions is introduced which, on first sight, is seemingly weaker in its premises than the one discussed in the previous section. While keeping~\ref{L0} and~\ref{L1}, the axiom~\ref{L2} gets replaced by the condition

\begin{enumerate}[label=(L\theenumi')]
\setcounter{enumi}{1}
\item
\label{L2p}
Denote by $\mathcal{W}_L$ the class of morphisms in $\mathcal{C}$ generated by $\mathcal{W}$ and all split monos in $\mathcal{C}$. Then given $w\in\mathcal{W}$ and parallel morphisms $f_1,f_2$ such that $f_1w=f_2w$, there exists $w'\in\mathcal{W}_L$ such that $w'f_1=w'f_2$.
\beq
\xymatrix{{\bullet}\ar@{~>}[r]^w & {\bullet}\ar@/^/[r]^{f_1}\ar@/_/[r]_{f_2} & {\bullet}\ar@{~>}[r]^{w'} & {\bullet}}
\eeq
\end{enumerate}

\begin{prop}
\label{higsonfractions}
Given $w'\in\mathcal{W}_L$, we can find $k\in\mathcal{C}$ such that $kw'\in\mathcal{W}$.
\end{prop}

\begin{proof}
Let us consider the cases how $w'$ might look like, one by one and in increasing order of difficulty. If already $w'\in\mathcal{W}$, we are done since we can take $k=\mathrm{id}_{\mathrm{cod}(w')}$. If $w'=m\widehat{w}$, where $m$ is a split mono and $\widehat{w}\in\mathcal{W}$, we can take $k$ to be a left-inverse of $m$, so we are done as well. 

The only non-trivial type of situation occurs $w'$ is a composition of morphisms in $\mathcal{W}$ and split monos such that morphisms of $\mathcal{W}$ come after split monos. The prototype for this situation is a morphism like $w'=\widehat{w}m$, with $\widehat{w}\in\mathcal{W}$ and $m$ a split mono. By assumption, $m$ has a left inverse $e$, which means $em=\mathrm{id}$. Now apply~\ref{L1} to the pair $\widehat{w},e$,
\beq
\xymatrix@+5pt{{\bullet}\ar@{~>}[r]^{\widehat{w}}\ar[d]^e & {\bullet}\ar[d]^k\\
{\bullet}\ar@{~>}[r]_{\widetilde{w}}\ar@/^1pc/[u]^m\ar@{~>}[ru]|{\phantom{w}w'} & {\bullet}}
\eeq
which gives the morphism $k$ and some morphism $\widetilde{w}\in\mathcal{W}$. The commutativity assertion of~\ref{L1} states in this case $\widetilde{w}e=k\widehat{w}$, so after composing with $m$ on the right we have $\widetilde{w}=\widetilde{w}em=k\widehat{w}m=kw'\in\mathcal{W}$.

Now for the general case. By definition of $\mathcal{W}_L$ and~\ref{L0}, our $w'$ is of the form
\beqn
w'=w_nm_n\cdots w_1m_1
\eeqn
where the $m_j$ are split monos and $w_j\in\mathcal{W}$. Starting from the left, we can iteratively apply the previous argument and use~\ref{L0} to compose the morphisms in $\mathcal{W}$ to a single morphism in $\mathcal{W}$, until we have only a single morphism in $\mathcal{W}$ left.
\end{proof}

\begin{cor}
\begin{enumerate}[label=(\alph{enumi})]
\label{splitfractions}
\item Given~\ref{L0} and~\ref{L1}, the assertions~\ref{L2} and~\ref{L2p} are equivalent.
\item Two roofs $(f_1,w_1)$ and $(f_2,w_2)$ are equivalent if and only if there is a diagram
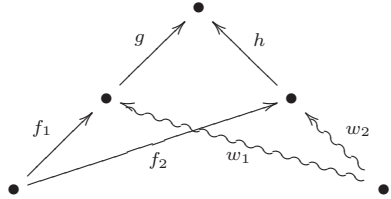
\beq
\xymatrix{ & & {\bullet} & & \\
& {\bullet}\ar[ur]^g & & {\bullet}\ar[ul]_h & \\
{\bullet}\ar[ur]^{f_1}\ar[urrr]_{f_2} & & & & {\bullet} \ar@{~>}[ulll]^{w_1} \ar@{~>}[ul]_{w_2}}
\eeq
where now we only demand commutativity and $hw_2\in\mathcal{W}_L$ (instead of $hw_2\in\mathcal{W}$).
\end{enumerate}
\end{cor}

\begin{proof}
These are both immediate consequences of the previous proposition.
\end{proof}

\begin{rem}
As already noticed in~\cite[1.2.4]{Hig}, when~\ref{L0} holds the axiom~\ref{L1} is in fact equivalent to the variant where $w\in\mathcal{W}_L$:

\begin{enumerate}[label=(L\theenumi')]
\item\label{L1p} Given any $w\in\mathcal{W}_L$ and an arbitrary morphism $f$ with $\mathrm{dom}(f)=\mathrm{dom}(w)$, we can find $w'\in\mathcal{W}$ and some morphism $f'$ with $\mathrm{cod}(f')=\mathrm{cod}(w')$, such that the diagram
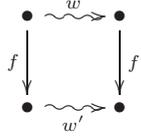
\beq
\xymatrix{{}{\bullet}\ar@{~>}[r]^w\ar[d]_f & {}{\bullet}\ar[d]^{f'} \\
{\bullet}\ar@{~>}[r]_{w'} & {\bullet}}
\eeq
commutes.
\end{enumerate}

Clearly,~\ref{L1} is trivially implied by this. For the other implication direction, by remark~\ref{generated} it is sufficient to show that~\ref{L1p} holds if $w$ is any split mono as in the diagram
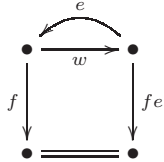
\beq
\xymatrix@+5pt{{\bullet}\ar[r]_w\ar[d]_f & {\bullet}\ar[d]^{fe}\ar@/_1pc/[l]_e\\
{\bullet}\ar@{=}[r] & {\bullet}}
\eeq
with $e$ some left-inverse of $w$. Then by~\ref{L0} we have $w'\equiv\mathrm{id}_{\mathrm{cod}(f)}\in\mathcal{W}$, so together with $f'\equiv fe$ this does the job; commutavity $few=f$ holds since $e$ is left-inverse to $w$.
\end{rem}

\section{Additive categories of fractions}
\label{sectionadditive}

Often, the working mathematician deals with additive categories. In particular, they may want to do localization completely within the framework of additive categories. In other words, given an additive category $\mathcal{C}$ and a class of ``moral isomorphisms'' $\mathcal{W}$ in $\mathcal{C}$, is there an additive category $\mathcal{C}[\mathcal{W}^{-1}]$ and an additive localization functor $Loc:\mathcal{C}\ra\mathcal{C}[\mathcal{W}^{-1}]$ which maps $\mathcal{W}$ to isos and is the universal additive functor with this property? And if yes, how can this localization be constructed?

For simplicity, we consider only the category of fractions case. Then, in fact, the localization constructed in theorem~\ref{loc} already \emph{is} additive. Inituitively, the reason is that one can find a ``common denominator'' for pairs of roofs representing parallel morphisms in $\mathcal{C}[\mathcal{W}^{-1}]$. The purpose of this section is to turn this intuitive explanation into a formal proof.

In the following, $\mathcal{C}$ is an additive category, and $\mathcal{W}\subseteq\mathcal{C}$ is a class of morphisms satisfying~\ref{L0},~\ref{L1} and~\ref{L2} (or the alternatives~\ref{L1p} and~\ref{L2p} discussed in the previous section).

We start by constructing ``common denominators'' and using them to define an addition operation on equivalence classes of roofs.

Given parallel roofs $(f_1,w_1)$ and $(f_2,w_2)$, we apply~\ref{L1} to the pair $w_1$, $w_2$ and obtain a diagram
\beqn
\begin{split}
\label{cd}
\xymatrix{&&{\bullet} &&\\
& {\bullet}\ar[ru]^g && {\bullet}\ar@{~>}[lu]_{\widetilde{w}}&\\
{\bullet}\ar[ru]^{f_1}\ar[rrru]_(.4){f_2} &&&& {\bullet}\ar@{~>}[lu]_{w_2}\ar@{~>}[lllu]^(.4){w_1}}
\end{split}
\eeqn
which commutes only in the sense that $gw_1=\widetilde{w}w_2$. There is an equivalence $(f_1,w_1)\sim(gf_1,\widetilde{w}w_2)$, and similarly $(f_2,w_2)\sim(\widetilde{w}f_2,\widetilde{w}w_2)$. Thus we have identified $\widetilde{w}w_2$ as a ``common denominator''. Now we can define the sum of $(f_1,w_1)$ and $(f_2,w_2)$ as
\beqn
\label{defnadd}
(f_1,w_1)+(f_2,w_2)\equiv (gf_1+\widetilde{w}f_2,\widetilde{w}w_2) \:.
\eeqn
It needs to be checked that the class of $(gf_1+\widetilde{w}f_2,\widetilde{w}w_2)$ does not depend on the particular choice of $g$ and $\widetilde{w}$. Thanks to lemma~\ref{l1unique}, the class $[(g,\widetilde{w})]$ is well-defined by $w_1$ and $w_2$. Now if $(g',\widetilde{w}')$ is another choice connected to $(g,\widetilde{w})$ by an elementary equivalence $h$, then we have the diagram
$$
\xymatrix{&&{\bullet}\ar@/^1pc/[rr]^h && \bullet \\
& {\bullet}\ar[ru]^g \ar[rrru]_{g'} &&&& {\bullet}\ar@{~>}[lllu]^{\widetilde{w}} \ar@{~>}[lu]_{\widetilde{w}'} &\\
{\bullet}\ar[ru]^{f_1}\ar[rrrrru]_(.4){f_2} &&&&&& {\bullet}\ar@{~>}[lu]_{w_2}\ar@{~>}[lllllu]^(.4){w_1}}
$$
which commutes only in the sense that $gw_1=\widetilde{w}w_2$, $g'w_1=\widetilde{w}'w_2$, $hg=g'$ and $h\widetilde{w}=\widetilde{w}'$. The equation $g'f_1+\widetilde{w}'f_2=h(gf_1+\widetilde{w}f_2)$ shows that $h$ likewise implements an equivalence
$$
(g'f_1+\widetilde{w}'f_2,\widetilde{w}'w_2)\sim (gf_1+\widetilde{w}f_2,\widetilde{w}w_2) \:,
$$
as was to be shown.

While it has been proven that the definition~(\ref{defnadd}) produces a well-defined class of roofs from every pair of roofs, it is still unclear whether the sum depends on the particular representatives of the summands or only on their classes.

\begin{lem}
The class of the sum only depends on the classes of the summands and not on the particular representatives.
\end{lem}

\begin{proof}
Still using the same notation, it is sufficient to consider an elementary equivalence between $(f_1,w_1)$ and some $(f'_1,w'_1)$:
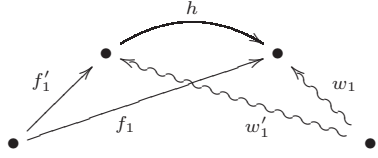
\beq
\xymatrix{&{\bullet}\ar@/^1pc/[rr]^h && {\bullet}&\\
{\bullet}\ar[ru]^{f'_1}\ar[rrru]_(.4){f_1} &&&&{\bullet}\ar@{~>}[lllu]^(.4){w'_1}\ar@{~>}[lu]_{w_1}}
\eeq
Then taking the common denominator of $(f_1,w_1)$ and $(f_2,w_2)$ as above yields the partially commutative diagram
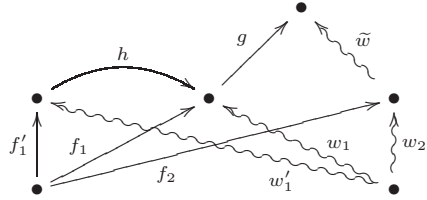
\beq
\xymatrix{&&&{\bullet}\\
{\bullet}\ar@/^1pc/[rr]^h && {\bullet}\ar[ru]^g && {\bullet}\ar@{~>}[lu]_{\widetilde{w}}\\
{\bullet}\ar[u]^{f'_1}\ar[rru]^(.3){f_1}\ar[rrrru]_(.35){f_2} &&&&{\bullet}\ar@{~>}[llllu]^(.3){w'_1}\ar@{~>}[llu]_(.35){w_1}\ar@{~>}[u]_{w_2}}
\eeq
Now the sum of $(f_1,w_1)$ and $(f_2,w_2)$ is the class of
\beqn
\label{sum1}
(gf_1+\widetilde{w}f_2,\widetilde{w}w_2)
\eeqn
while the sum of $(f'_1,w'_1)$ and $(f_2,w_2)$ is the class of
\beqn
(ghf'_1+\widetilde{w}f_2,\widetilde{w}w_2)
\eeqn
which coinides with~(\ref{sum1}) by commutativity of the diagram.
\end{proof}

\begin{thm}
\label{additiveloc}
Suppose $\mathcal{C}$ is additive and allows a calculus of left fractions with respect to $\mathcal{W}$. Then the category of fractions $\mathcal{C}[\mathcal{W}^{-1}]$ is additive.
\end{thm}

\begin{proof}
A category is additive if it is preadditive, has a zero object, and has a biproduct for every pair of objects.

It was already shown how to add equivalence classes of roofs and that this operation is well-defined. Its associativity can be seen from a diagram of the symbolic form
$$
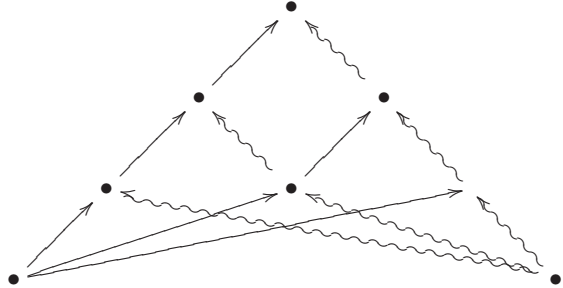
\xymatrix{&&& {\bullet} &&&\\
&& \bullet\ar[ru] && \bullet\ar@{~>}[ul] \\
& {\bullet}\ar[ru] && {\bullet}\ar@{~>}[lu]\ar[ru] && \ar@{~>}[ul] \\
{\bullet}\ar[ru]\ar[rrru]\ar[rrrrru] &&&&&& {\bullet}\ar@{~>}[lu]\ar@{~>}[lllu]\ar@{~>}[lllllu] }
$$
Its commutativity is evident from~(\ref{cd}) by realizing that $g$ and $\widetilde{w}$ play identical r\^oles in~(\ref{cd}): it is not relevant that $\widetilde{w}\in\mathcal{W}$, but only that $\widetilde{w}w_2\in\mathcal{W}$. Neutral elements of the addition operation are given by the equivalence classes $[(0,\mathrm{id})]$. An additive inverse of $[(f,w)]$ is $[(-f,w)]$. Hence the category of fractions is preadditive.

The localization functor was defined as $Loc:f\mapsto(f,\mathrm{id})$. For parallel morphisms $f,g\in\mathcal{C}$, adding roofs gives $[(f,\mathrm{id})]+[(g,\mathrm{id})]=[(f+g,\mathrm{id})]$; in other words, $Loc$ is additive. In particular, $Loc$ maps biproduct diagrams to biproduct diagrams. Then since the functor is surjective on objects, $\mathcal{C}[\mathcal{W}^{-1}]$ has biproducts. Any null object of $\mathcal{C}$ also is a null object in $\mathcal{C}[\mathcal{W}^{-1}]$. All in all, this makes $\mathcal{C}[\mathcal{W}^{-1}]$ additive.
\end{proof}

\begin{rem}
In an additive category, one can obviously replace the axiom~\ref{L2} by the slightly simpler requirement
\begin{enumerate}[label=(L\theenumi'')]
\setcounter{enumi}{1}
\item 
Given $w\in\mathcal{W}$ and a morphism $f$ such that $fw=0$, there exists $w'\in\mathcal{W}$ (or $w'\in\mathcal{W}_L$) such that $w'f=0$.
\beq
\xymatrix{{\bullet}\ar@{~>}[r]^w & {\bullet}\ar[r]^f & {\bullet}\ar@{~>}[r]^{w'} & {\bullet}}
\eeq
\end{enumerate}
\end{rem}

\bibliographystyle{plain}
\bibliography{cats_of_fractions}

\end{document}